\documentclass[a4paper]{amsart}

\usepackage{enumerate}
\usepackage{amsmath}%
\usepackage{amsfonts}%
\usepackage{amssymb}%
\usepackage{graphicx}
\usepackage{mathrsfs}
\usepackage{hyperref}
\usepackage[margin=1in]{geometry}
\newtheorem{theorem}{Theorem}
\theoremstyle{plain}

\newtheorem{conjecture}[theorem]{Conjecture}

\newtheorem{lemma}[theorem]{Lemma}

\numberwithin{equation}{section}
\numberwithin{theorem}{section}
\numberwithin{case}{section}

\numberwithin{subcase}{case}
\def \a{\alpha}
\def \e{\epsilon}
\def \r{\gamma}

\begin{document}

\title{On vertex-disjoint paths in regular graphs}
\author{Jie Han}
\thanks{The author is supported by FAPESP (2013/03447-6, 2014/18641-5)}

\address{Instituto de Matem\'{a}tica e Estat\'{\i}stica, Universidade de S\~{a}o Paulo, \\
Rua do Mat\~{a}o 1010, 05508-090, S\~{a}o Paulo, Brazil. \\
Email: jhan@ime.usp.br}

\begin{abstract}
Let $c\in (0, 1]$ be a real number and let $n$ be a sufficiently large integer.
We prove that every $n$-vertex $c n$-regular graph $G$ contains a collection of $\lfloor 1/c \rfloor$ paths whose union covers all but at most $o(n)$ vertices of $G$.
The constant $\lfloor 1/c \rfloor$ is best possible when $1/c\notin \mathbb{N}$ and off by $1$ otherwise.
Moreover, if in addition $G$ is bipartite, then the number of paths can be reduced to $\lfloor 1/(2c) \rfloor$, which is best possible.
%\keywords{Perfect packing, computational complexity, absorbing method}
\end{abstract}

\maketitle

\section{Introduction}
Paths and cycles are fundamental objects in graph theory.
% and embedding paths or cycles have been the central area of graph theory.
The \emph{path cover number} is the minimum number of vertex-disjoint paths whose union covers all vertices of $G$.
Note that we allow paths of length $0$ (single vertices) in the definition above.
Trivially the path cover number of a graph $G$ is upper bounded by the independence number of $G$, because the set of the (arbitrary one out of the two) end vertices of the paths in a minimal path cover form an independent set in $G$. 
It is evident that for general graphs, determining the path cover number is NP-hard, because deciding if the path cover number equals $1$ is equivalent to the decision problem for Hamiltonian path, which is NP-complete.
For bounds on the path cover number for general graphs, see e.g.~\cite{Hartman, Ore61}.
For regular graphs, Magnant and Martin~\cite{MaMa09} made the following conjecture and confirmed it for $0\le k\le 5$.

\begin{conjecture}\label{conj}
If $G$ be a $k$-regular graph of order $n$, then the path cover number of $G$ is at most $n/(k+1)$.
\end{conjecture}

If true, the bound in Conjecture~\ref{conj} would be tight as seen by disjoint copies of complete graph $K_{k+1}$ (if $n \equiv j$ modulo $k+1$ and $j\neq 0$, then we change one copy of $K_{k+1}$ to a copy of $K_{k+1+j}$).
By the celebrated Dirac theorem on Hamiltonian paths~\cite{Dirac}, Conjecture~\ref{conj} is true for $k\ge (n-1)/2$.
As far as we know, Conjecture~\ref{conj} is open for all other cases.
To provide more evidence on the validity of the conjecture, in this note we prove the following result for dense regular graphs.

\begin{theorem}\label{thm1}
For any $c, \a>0$, there exists $n_0\in \mathbb{N}$ such that the following holds for every integer $n\ge n_0$.
\begin{enumerate}
\item Every $\lceil c n\rceil$-regular graph of order $n$ contains a collection of at most $\lfloor 1/c \rfloor$ vertex-disjoint paths whose union covers all but $\a n$ vertices.
\item Every bipartite $\lceil c n\rceil$-regular graph of order $n$ contains a collection of at most $\lfloor 1/(2c) \rfloor$ vertex-disjoint paths whose union covers all but $\a n$ vertices.
\end{enumerate}
\end{theorem}

Note that Part (2) of the theorem corresponds to the bipartite version of Conjecture~\ref{conj}: if $G$ is a bipartite $k$-regular graph of order $n$, then the path cover number of $G$ can be as large as $n/(2k)$, as seen by the vertex-disjoint copies of $K_{k,k}$.
Note that 
\[
\frac{-2}{c(\lceil c n\rceil+1)}\le \frac{n}{\lceil c n\rceil+1} - \frac1c = \frac{c n - \lceil c n\rceil-1}{c(\lceil c n\rceil+1)} < 0.
\]
So when $n$ is large, if $1/c\notin \mathbb N$, then $\lfloor \frac{n}{\lceil cn\rceil+1}\rfloor = \lfloor 1/c \rfloor$, i.e., the number of paths in Theorem~\ref{thm1} matches the quantity in Conjecture~\ref{conj}; however if $1/c \in \mathbb N$, then the quantity $\lfloor 1/c \rfloor$ is off by $1$.
On the other hand, the quantity $\lfloor 1/(2c) \rfloor$ in Part (2) is optimal.

% Note that the following cleaner statement can be obtained by the same proof.
% \begin{theorem}
% For any $c, \a>0$, there exists an integer $n_0$ such that the following holds for every integer $n\ge n_0$.
% \begin{enumerate}
% \item Every $k$-regular graph of order $n$ with $k\ge c n$ contains a collection of at most $\lfloor \frac{n}{k+1} \rfloor+1$ vertex-disjoint paths whose union covers all but $\a n$ vertices.
% \item Every bipartite $k$-regular graph of order $n$ with $k\ge c n$ contains a collection of at most $\lfloor \frac{n}{2k} \rfloor$ vertex-disjoint paths whose union covers all but $\a n$ vertices.
% \end{enumerate}
% \end{theorem}

At last, we remark that the bound in Conjecture~\ref{conj} is not tight if we restrict the problem on connected regular graphs, see~\cite{Reed96} for connected cubic graphs.

%The rest of the paper is organized as follows. 
%In Section 2 we introduce the main ideas of the proof of Theorem~\ref{thm:main}. 
%We introduce Lemma~\ref{lem:PL} and Theorem \ref{thm:PT} in Section 3 and use them to prove Theorem~\ref{thm:main}.
%We prove Theorem~\ref{thm:PT} in Section 4. 
%We show the rest of the proofs in the appendix.

%\noindent {\bf Notation.}
%For any graph $G$, we write $|G|$ for its order, $e(G)$ for its number of edges, $\delta(G)$ for its minimum degree, $\chi(G)$ for its chromatic number and $\chi_{cr}(G)$ for its critical chromatic number defined in Section~1.
%Given $A\subseteq V(G)$, let $G[A]$ be the induced subgraph of $G$ on $A$; given $A, B\subseteq V(G)$ with $A\cap B=\emptyset$, let $G[A, B]$ be the induced bipartite subgraph of $G$ with parts $A$ and $B$.
%Given integers $k\ge 2$ and $a_1, \dots, a_k$, let $K_{a_1,\dots, a_k}$ be the complete $k$-partite graph with color class sizes $a_1, \dots, a_k$.
%Throughout this paper, $x\ll y$ means that for any $y> 0$, there exists $x_0> 0$ such that for any $0<x\le x_0$ the following statement holds.
%Hierarchies of other lengths are defined in the obvious way.

\section{A weaker result}

We first prove the following weaker result.
For reals $a, b, c$, we write $a=(1\pm b)c$ if there exists a real $x\in (1-b, 1+b)$ such that $a=xc$.

\begin{theorem}\label{thm2}
Given any reals $c, \a>0$, there exists $\e>0$ and integer $C>0$ such that the following holds for sufficiently large integer $n$.
Let $G$ be a graph of order $n$ such that $\deg(v)=(1\pm\e)c n$ for every $v\in V(G)$.
Then there exists a collection of $C$ vertex-disjoint cycles in $G$ whose union covers all but $\a n$ vertices of $G$.
\end{theorem}

Our main tools for embedding the cycles are the Regularity Lemma of Szemer\'edi~\cite{Sze} and the Blow-up Lemma of Koml\'os et al.~\cite{Blowup}.
For any two disjoint vertex-sets $A$ and $B$ of a graph $G$, the density of $A$ and $B$ is defined as $d(A, B):=e(A, B)/(|A||B|)$, where $e(A, B)$ is the number of edges with one end vertex in $A$ and the other in $B$.
Let $\e$ and $\delta$ be two positive real numbers.
The pair $(A, B)$ is called \emph{$\e$-regular} if for every $X\subseteq A$ and $Y\subseteq B$ satisfying $|X|>\e |A|$, $|Y|>\e |B|$, we have $|d(X, Y) - d(A, B)|<\e$.
Moreover, the pair $(A, B)$ is called \emph{$(\e, \delta)$-super-regular} if $(A, B)$ is $\e$-regular and $\deg_B(a) > \delta |B|$ for all $a\in A$ and $\deg_A(b) > \delta |A|$ for all $b\in B$.

\begin{lemma}[Regularity Lemma -- Degree Form]\label{RL}
For every $\e>0$ there is an $M=M(\e)$ such that for any graph $G=(V, E)$ and any real number $d\in [0,1]$, there is a partition of the vertex set $V$ into $t+1$ clusters $V_0, V_1,\dots, V_t$, and there is a subgraph $G'$ of $G$ with the following properties:
\begin{itemize}
\item $t\le M$,
\item $|V_i|\le \e |V|$ for $0\le i\le t$ and $|V_1|=|V_2|=\cdots=|V_t|$,
\item $\deg_{G'}(v) > \deg_G (v) - (d+\e) |V|$ for all $v\in V$,
\item $G'[V_i]=\emptyset$ for all $i$,
\item each pair $(V_i, V_j)$, $1\le i< j\le t$, is $\e$-regular with $d(V_i, V_j)=0$ or $d(V_i, V_j)\ge d$ in $G'$.
\end{itemize}
\end{lemma}

The Blow-up Lemma allows us to regard a super regular pair as a complete bipartite graph when embedding a graph with bounded degree.
Since we will always use it to embed a cycle, we state it in the following special form.

\begin{lemma}\label{BL}
For every $\delta>0$, there exists an $\e>0$ such that the following holds for sufficiently large integer $N$.
Let $(X, Y)$ be an $(\e, \delta)$-super-regular pair with $|X|=|Y|=N$.
Then $(X, Y)$ contains a spanning cycle (a cycle of length $2N$).
\end{lemma}

A \emph{fractional matching} is a function $f$ that assigns to each edge of a graph a real number in $[0,1]$ so that, for each vertex $v$, we have $\sum f(e)\le 1$ where the sum is taken over all edges incident to $v$.
The \emph{fractional matching number $\mu_f(G)$} of a graph $G$ is the supremum of $\sum_{e\in E(G)} f(e)$ over all fractional matchings $f$.
We use the following so-called `{fractional Berge-Tutte formula}' of Scheinerman and Ullman~\cite[Theorem 2.2.6]{ScUl97}. 
Note that it is also proved in~\cite{ScUl97} that (see Theorem~2.1.5) in a graph $G$, the maximum fractional matching, i.e., with weight $\mu_f(G)$, can be achieved with weights only chosen from $\{0, 1/2, 1\}$.

\begin{theorem}\cite{ScUl97}\label{thm:2mat}
For any graph $G$, 
\[
\mu_f(G) = \frac12 \left( |V(G)| - \max\{i(G-S) - |S|\} \right),
\]
where $i(X)$ denotes the number of isolated vertices in $G[X]$, and the maximum is taken over all $S\subseteq V(G)$. 
\end{theorem}

\begin{proof}[Proof of Theorem~\ref{thm2}]
Given $c, \a>0$, let $d=\a c/9$.
We apply Lemma~\ref{BL} with $\delta=d/2$ and obtain $\e_1>0$.
Let $\e=\min\{\e_1, d/6, 3d/(2c)\}$.
We then apply Lemma~\ref{RL} with $\e$ and obtain $M=M(\e)$. Let $n\in \mathbb{N}$ be sufficiently large.
Let $G=(V, E)$ be a graph of order $n$ such that $\deg(v)=(1\pm\e)c n$ for every $v\in V$.
We apply the Regularity Lemma (Lemma~\ref{RL}) on $G$ with the constants $\e, d$ chosen as above and obtain a partition of $V$ into $V_0, V_1, \dots, V_t$ for some $t\le M$, and a subgraph $G'$ of $G$ with the properties as described in Lemma~\ref{RL}.
By moving at most one vertex from each $V_i$, $i\in [t]$ to $V_0$, we can assume that $m:=|V_i|$ is even.
Thus we have $|V_0|\le \e n + t\le 2\e n$.
Now for any $v\in V$,
\[
\deg_{G'}(v) - |V_0| > \deg_G(v) - (d+\e)n - 2\e n \ge (c - 2d)n.
\]

Let $\beta = 3d/c$.
Let $H$ be the graph on $[t]$ such that $ij\in E(H)$ if and only if $d(V_i, V_j) \ge d$.
We first assume that there exists a set $S\subseteq [t]$, such that $i(H-S) - |S| \ge \beta t$.
In particular, let $T$ be the collection of $|S|+\beta t$ isolated vertices in $H-S$.
Thus we have 
\[
e_G(T, S) \ge |T|m (\deg_{G'}(v) - |V_0|) \ge |T| m (c-2d)n.
\]
However, by averaging, this implies that there exists a vertex $v\in V$ such that
\[
\deg_G(v) \ge \deg_{G'}(v)\ge \frac{|T| m (c-2d)n}{|S| m} \ge \frac{t}{t - \beta t} (c-2d)n > (1+\e)c n,
\]
by the definition of $\beta$ and $\e$, a contradiction.
Thus we have $i(H-S) - |S| \le \beta t$ for any $S\subseteq [t]$.
So by Theorem~\ref{thm:2mat}, we get $\mu_f(H)\ge (1-\beta) t/2$.
Moreover, there exists a fractional matching $f$ such that $\sum_{e\in E(H)} f(e) = \mu_f(H)\ge (1-\beta) t/2$ and $f(e)\in \{0, 1/2, 1\}$ for every edge $e\in E(H)$.

For each $i\in [t]$ we arbitrarily split $V_i$ into $V_i^1$ and $V_i^2$ each of size $m/2$.
Thus the existence of $f$ implies that we can partition $V\setminus V_0$ into at least $(1-\beta) t/2$ pairs of sets each of form $(V_i^a, V_j^b)$ with density at least $d$, where $i, j\in [t]$, $i\neq j$ and $a, b\in [2]$, and a set of at most $2\beta t\cdot m$ vertices.
Note that here (to simplify the argument) even if an edge $i j \in E(H)$ receives weight $1$, we still split it, e.g., as $(V_i^1, V_j^1)$ and $(V_i^2, V_j^2)$.
Thus every vertex of $H$ is in at most two pairs so there are at most $t$ pairs. 
%Moreover, note that each such pair is $2\e$-regular of density at least $d-\e$.

We will show that each such pair contains a cycle that covers all but at most $2\e m$ vertices.
Indeed, fix any pair $(V_i^a, V_j^b)$, let $A$ be the set of vertices in $V_i^a$ whose degree to $V_j^b$ is less than $(d-\e)|V_j^b|$.
Since $d(A, V_j^b) < d-\e$ and $|V_j^b|>\e m$, the regularity of $(V_i, V_j)$ implies that $|A|\le \e m $.
Similarly let $B$ be the set of vertices in $V_j^b$ whose degree to $V_i^a$ is less than $(d-\e)|V_i^a|$ and we have $|B|\le \e m$.
Let $A'\supseteq A$ and $B'\supseteq B$ be arbitrary subsets of $V_i^a$ and $V_j^b$, respectively, of size exactly $\e m$.
Now let $X=V_i^a\setminus A'$ and $Y=V_j^b\setminus B'$, we get that $(X, Y)$ is $(\e, d-3\e)$-super-regular with density at least $d-3\e$, and $|X|=|Y|=m-\e m$.
Since $d-3\e\ge d/2$, by Lemma~\ref{BL}, $(X, Y)$ contains a spanning cycle and we are done.

Let $C=M$.
Thus we obtain a set of at most $t\le M=C$ vertex-disjoint cycles in $G$ that covers all but at most
\[
t\cdot 2\e m+ |V_0| + 2\beta t \cdot m\le 3\beta n  = \a n
\]
vertices, completing the proof.
\end{proof}

\section{Proof of Theorem~\ref{thm1}}
In the proof of Theorem~\ref{thm1} we use the trick of a `reservoir lemma' from~\cite{RRS06, RRS08}.
Roughly speaking, we will reserve a random set $R$ of vertices at the beginning of the proof, and use them to connect the paths returned by applying Theorem~\ref{thm2} on $G-R$.
We first recall the following Chernoff's bounds (see, e.g.,~\cite{AlSp}) for binomial random variables and for $x>0$:
\begin{align*}
&\mathbb P[Bin(n', \zeta)\ge n'\zeta + x]< e^{-x^2/(2 n'\zeta + x/3)} \\
&\mathbb P[Bin(n', \zeta)\le n'\zeta - x]< e^{-x^2/(2 n'\zeta)}. 
\end{align*}

\begin{lemma}\label{lemR}
Given any $c, \r, \e>0$, the following holds for sufficiently large integer $n$.
Let $G$ be a $\lceil c n\rceil$-regular graph of order $n$.
Then there exists a set $R\subseteq V(G)$ such that $|R|=(1\pm\e)\r n$ and every vertex of $G$ has degree $(1\pm\e)c \r n$ in $R$.
\end{lemma}

\begin{proof}
We select the set $R$ by including each vertex of $G$ independently and randomly with probability~$\r$.
Note that $|R|$ and $\deg(v, R)$ for each $v\in V(G)$ are both binomial random variables with expectation $\r n$ and $\r \lceil c n\rceil$, respectively.
By Chernoff's bounds, we get
\begin{align*}
&\mathbb P[|R|>(1+\e)\r n] < e^{-\e^2 \r n/3},\, \mathbb P[|R|<(1-\e)\r n] < e^{-\e^2 \r n/3}, \\
&\mathbb P[x_v>(1+\e)c \r n] < e^{-\e^2 c \r n/3}, \, \mathbb P[x_v<(1-\e)c \r n] < e^{-\e^2 c \r n/3},
\end{align*}
where $x_v:=\deg(v, R)$ for all $v\in V(G)$.
Since $(2n+2)e^{-\e^2 c \r n/3}<1$ because $n$ is large enough, there is a choice of $R$ with the desired properties.
\end{proof}

\begin{proof}[Proof of Theorem~\ref{thm1}]
Given $c, \a\in (0,1)$, let $\r= \a/4$.
We apply Theorem~\ref{thm2} with $c$ and $\a/2$ in place of $\a$ and obtain $\e_1$ and $C\in \mathbb N$.
Let $\e = \min\{\e_1, ((\lfloor 1/c \rfloor+1)c-1)/3\}$.
Let $G$ be a $\lceil c n\rceil$-regular graph of order $n$.
We first pick the set $R$ by Lemma~\ref{lemR}.
Let $G_1=G-R$ and $n_1 = n - |R|$. 
Thus for every vertex $v\in V(G_1)$, we know that $\deg_{G_1}(v)=\lceil c n\rceil - (1\pm\e)c\r n$. 
Since $|R|=(1\pm\e)\r n$ we know that $\deg_{G_1}(v)= (1\pm \e)c n_1$.
Indeed, for the upper bound we have
\[
\deg_{G_1}(v)\le c(n_1+|R|) +1 - (1-\e) c\r n \le c n_1 +1 + 2 \e c \r n \le (1+\e)c n_1,
\]
where in the last inequality we use $n< 2n_1$ and $\r\le 1/4$; and the lower bound can be shown similarly.
By applying Theorem~\ref{thm2} with $\a/2$ in place of $\a$, we obtain a collection of at most $C$ vertex-disjoint paths whose union covers all but at most $\a|V(G_1)|/2\le \a n/2$ vertices of $G_1$.

Next we iteratively use the property of $R$ to connect some pair of paths.
We first explain the general case.
Suppose there are at least $\lfloor 1/c \rfloor+1$ paths left. Indeed, let $v_1,v_2,\dots,v_{\lfloor 1/c \rfloor+1}$ be the (arbitrary one out of the two) ends of the $\lfloor 1/c \rfloor+1$ paths.
Note that throughout the iteration there are at most $C$ vertices in $R$ that have been used for connecting and thus removed from $R$.
So for each $i$, by $\deg(v_i, R)-C\ge (1-\e) c \r n-C\ge (1-2\e) c |R|$, we have
\[
(\lfloor 1/c \rfloor+1)(\deg(v_i, R)-C) \ge (1+3\e) (1-2\e) |R| > |R|,
\]
by the definition of $\e$.
Thus there exist two vertices $v_i, v_j$ which have a common neighbor $w$ in $R$ so that we can connect the corresponding two paths by $w$.
At the end, we obtain a collection of at most $\lfloor 1/c \rfloor$ vertex-disjoint paths whose union covers all but at most $\a n/2 + (1+\e)\r n\le \a n$ vertices in $G$.

Second, assume that there are at least $\lfloor 1/(2c) \rfloor+1$ paths left and in addition that $G$ is bipartite with bipartition $X$ and $Y$.
Note that since $G$ is regular we have $|X|=|Y|=n/2$ (so in particular $n$ must be even).
Fix $\lfloor 1/(2c) \rfloor+1$ paths.
By throwing away at most one vertex from each path we can assume that each path has exactly one end vertex in $X$ and one in $Y$.
Let $v_1,v_2,\dots,v_{\lfloor 1/(2c) \rfloor+1}$ be the end vertices in $X$.
By the similar calculation, we can find a vertex $w\in R\cap Y$ which connects some pair of paths.
At the end, we obtain a collection of at most $\lfloor 1/(2c) \rfloor$ vertex-disjoint paths whose union covers all but at most $\a n/2 + (1+\e)\r n + C^2\le \a n$ vertices in $G$, because the iteration has at most $C$ steps and in each step we threw away at most $C$ vertices from the current paths.
\end{proof}

%\section{Concluding Remarks}
%
%In view of our result, to prove Conjecture~\ref{conj} for dense regular graphs, it suffices to prove a so-called `absorbing lemma'.
%Roughly speaking, such a lemma would allow one to find a set of (short) absorbing paths in $G$, and if one applies the proof of Theorem~\ref{thm1} on the rest of the graph, then one can still guarantee an almost path cover by at most $\lfloor 1/c \rfloor$ paths (including the absorbing paths).
%The one should expect to `absorb' the uncovered vertices by the absorbing paths and obtain a path cover of $G$.

\section*{Acknowledgement}

The author would like to thank Yoshiharu Kohayahawa, Phablo Moura and Yoshiko Wakabayashi for discussions and valuable comments on the manuscript.

\bibliographystyle{plain}
\bibliography{Bibref}

\end{document}